\documentclass{amsart}
\usepackage{amssymb, amsmath, amsthm, graphicx, amscd, pxfonts, mathtools, tikz-cd}

\theoremstyle{plain}
\newtheorem*{main}{Main Theorem}

\newtheorem{thm}{Theorem}
\newtheorem{lm}[thm]{Lemma}

\newtheorem{prop}[thm]{Proposition}

\theoremstyle{definition}

\newtheorem{de}[thm]{Definition}
\newtheorem{ex}[thm]{Example}

\newtheorem{re}[thm]{Remark}

\DeclareMathOperator{\RR}{\mathbb{R}}
\DeclareMathOperator{\CC}{\mathbb{C}}

\DeclareMathOperator{\GL}{GL}
\DeclareMathOperator{\SL}{SL}
\DeclareMathOperator{\SU}{SU}
\DeclareMathOperator{\OO}{O}
\DeclareMathOperator{\SO}{SO}

\DeclareMathOperator{\g}{\mathfrak{g}}
\DeclareMathOperator{\gl}{\mathfrak{gl}}
\renewcommand{\k}{\kappa}
\renewcommand{\sl}{\mathfrak{sl}}
\DeclareMathOperator{\su}{\mathfrak{su}}
\renewcommand{\sp}{\mathfrak{sp}}

\DeclareMathOperator{\reg}{reg}
\DeclareMathOperator{\Tr}{Tr}
\renewcommand{\span}{\mbox{span}}

\DeclareMathOperator{\ad}{ad}
\DeclareMathOperator{\Sp}{Sp}

\DeclareMathOperator{\End}{End}

\newcommand{\diag}{\mathrm{diag}}

\begin{document}
\title{A note on ED degrees of group-stable subvarieties in polar representations} 
\author{Arthur Bik}
\address{Universit\"at Bern, Mathematisches Institut, Alpeneggstrasse 22,
3012 Bern}
\email{arthur.bik@math.unibe.ch}
\author{Jan Draisma}
\address{Universit\"at Bern, Mathematisches Institut, Sidlerstrasse 5,
3012 Bern, and Eindhoven University of Technology}
\email{jan.draisma@math.unibe.ch}
\begin{abstract}
In a recent paper, Drusvyatskiy, Lee, Ottaviani, and Thomas establish a ``transfer principle'' by means of which the Euclidean distance degree of an orthogonally-stable matrix variety can be computed from the Euclidean distance degree of its intersection with a linear subspace. We generalise this principle. 
\end{abstract}
\thanks{Both authors were partially supported by the NWO Vici grant
entitled {\em Stabilisation in Algebra and Geometry}.}
\maketitle

\section{Introduction} 

Fix a closed algebraic subvariety $X$ of a finite-dimensional complex vector space $V$ equipped with a non-degenerate symmetric bilinear form $\langle-|-\rangle\colon V\times V\rightarrow\CC$. Denote by $X^{\reg}$ the smooth locus in $X$.  Then for a sufficiently general {\em data point} $u \in V$ the number 
$$
\#\left\{\left.x \in X^{\reg} \right| u-x \perp T_x X\right\}
$$
of {\em ED critical points} for $u$ on $X$ is finite. Suppose that $V=\CC \otimes_{\RR} V_{\RR}$, the bilinear form is the complexification of a Euclidean inner product on $V_{\RR}$ and $X$ is the Zariski-closure of a real algebraic variety $X_{\RR}$ that has real smooth points, then this number is, for $u\in V$ sufficiently general, positive and independent of $u$ and is called the {\em Euclidean distance degree} (ED degree for short) of $X$ in $V$. See \cite{draisma-horobet-ottaviani-sturmfels-thomas}. Here, the ED degree counts the number of critical points in the smooth locus of $X$ of the distance function $d_u\colon X\rightarrow\CC$ sending $x \mapsto \langle u-x|u-x\rangle$.\bigskip

The goal of this note is to show that the ED degree of a variety $X$ with a suitable group action can sometimes be computed from that of a simpler variety $X_0$ obtained by slicing $X$ with a linear subspace of $V$.\bigskip

For the simplest example of this phenomenon, let $C\subseteq\CC^2$ be the unit circle with equation $x^2+y^2=1$ where $\CC^2$ is equipped with the standard form. The ED degree of $C$ equals $2$ and this is easily seen as follows. First, $C$ is $\OO_2$-stable where $\OO_2$ is the orthogonal group preserving the bilinear form. For all $u\in\CC^2$ and $g\in\OO_2$, the map $g$ restricts to a bijection between ED critical points on $C$ for $u$ and for $gu$. In particular, the number of ED critical points on $C$ for a sufficiently general point $u\in\CC^2$ equals that number for $gu$, for any choice of $g\in\OO_2$. We may assume that $u$ is not isotropic. Therefore, by choosing $g$ suitably, we may assume that $u$ lies on the horizontal axis. And then, since $u\not\perp T_p\OO_2p = T_pC$ for any point $p\in C$ not on the horizontal axis, the search for ED critical points is reduced to the search for such points on the intersection of $C$ with the horizontal axis. Clearly, both of the intersection points are critical.\bigskip

In the paper \cite{drusvyatskiy-lee-ottaviani-thomas}, a generalisation of this example is studied. They consider the
vector space $V=\CC^{n \times t}$ equipped with the trace bilinear form and with the group $G=\OO_n\times\OO_t$ acting by left and right multiplication. The variety $X$ is chosen as the Zariski-closure of an $(\OO_n(\RR)\times\OO_t(\RR))$-stable real algebraic variety $X_{\RR}$ in~$\RR^{n \times t}$. This ensures that $X$ is $G$-stable. The horizontal line is generalised to the $\min(n,t)$-dimensional space $V_0$ of diagonal matrices in $V$. They then prove that the ED degree of $X$ in $V$ equals the ED degree of $X_0:=X\cap V_0$ in $V_0$. In the paper, $X_0$ is defined in an {\em a priori} different manner, namely, as the Zariski-closure of the intersection of $X_{\RR}$ with $V_0$. That this is the same thing as the intersection of $X$ with $V_0$ is the content of \cite[Theorem 3.6]{drusvyatskiy-lee-ottaviani-thomas}, which is an application of the fact that the quotient map under a reductive (in fact, here finite) group sends closed, group-stable sets to closed sets.\bigskip

Note that, like the unit circle and the horizontal line from the first example, the variety $X$ and the subspace $V_0$ satisfy the following conditions:
\begin{itemize}
\item[(1)] For $v_0\in V_0$ sufficiently general, the vectorspace $V$ is the orthogonal direct sum of $V_0$ and $T_{v_0}Gv_0$. 
\item[(2)] The set $GX_0$ is dense in $X$. 
\end{itemize}
The tangent space $T_{v_0}Gv_0$ is equal to $\g v_0$ where
$\g$ is the Lie algebra of $G$, consists of all pairs
$(a,b)$ of skew-symmetric $n\times n$ and $t\times t$
matrices and acts by $(a,b)\cdot v=av-vb$ for all $v\in V$
and $(a,b)\in\g$. From the fact that the bilinear form
$\langle-|-\rangle$ is $G$-invariant, it follows that
$\langle(a,b)v|w\rangle+\langle v|(a,b)w\rangle=0$ for all
$v,w\in V$ and $(a,b)\in\g$. So condition (1) is equivalent
to the statement that if $v_0\in V_0$ is sufficiently
general, then $w\in V$ satisfies
$\Tr((aw)v_0^T)=\Tr((wb)v_0^T)=0$ for all skew-symmetric
$a\in\CC^{n\times n},b\in\CC^{t\times t}$ if and only if $w$
is a diagonal matrix. Using that symmetric matrices form the
orthogonal complement, with respect to the trace form, of
the space of skew-symmetric matrices, this is the content 
of \cite[Lemma 4.7]{drusvyatskiy-lee-ottaviani-thomas}. Condition (2) follows from the fact that the Zariski-dense subset of $X$ of {\em real} $n\times t$ matrices admit a singular value decomposition.\bigskip

We will generalize the result of \cite{drusvyatskiy-lee-ottaviani-thomas}
by showing that conditions (1) and (2) are sufficient for
establishing that the ED degree of $X$ in $V$ equals that of $X_0$
in $V_0$, and we will describe the orthogonal representations
that have such a subspace $V_0$---these turn out to be the {\em
polar representations} of the title.\bigskip

The remainder of the paper
is organized as follows.  In Section~\ref{section:mainresults} we
state our main results. Section~\ref{section:symmetric} showcases a
concrete optimization problem amenable to our techniques: given a real
symmetric matrix, find a closest symmetric matrix with prescribed
eigenvalues. In Section~\ref{section:realcomplex} we discuss the
relation between complex varieties to which our theorem applies, acted
upon by complex reductive groups; and their real counterparts acted
upon by compact Lie groups.  Section~\ref{section:proof} contains the
proof of our main theorem, and Section~\ref{section:testing} discusses
one possible approach for conclusively testing whether an orthogonal
representation is polar. Finally, in Section~\ref{section:examples}
we discuss some of the most important polar representations coming from the irreducible real polar representation found in~\cite{dadok}.

\section{Main results}\label{section:mainresults}

Let $V$ be a finite-dimensional complex vector space equipped with a non-degenerate symmetric bilinear form $\langle-|-\rangle\colon V\times V\rightarrow\CC$. Let $G$ a complex algebraic group and let $G \to O(V)$ be an orthogonal representation.

\begin{main}
Suppose that $V$ has a linear subspace $V_0$ such that, for sufficiently general $v_0 \in V_0$, the space $V$ is the orthogonal direct sum of $V_0$ and the tangent space $T_{v_0} G v_0$ of $v_0$ to its $G$-orbit. Let $X$ be a $G$-stable closed subvariety of $V$. Set $X_0:=X \cap V_0$ and suppose that $GX_0$ is dense in $X$. Then the ED degree of $X$ in $V$ equals the ED degree of $X_0$ in $V_0$.
\end{main}

\begin{re}
The condition that for sufficiently general $v_0 \in V_0$ the space $V$ is the orthogonal direct sum of $V_0$ and $T_{v_0} G v_0$ implies that the restriction of the form $\langle-|-\rangle$ to $V_0$ is non-degenerate and that $V_0$ and $T_{v_0}Gv_0$ are perpendicular for all $v_0\in V_0$.
\end{re}

\begin{re}
When $T_xX\cap (T_xX)^{\perp}=\{0\}$ for some $x\in
X^{\reg}$, then the ED degree of $X$ in $V$ is positive by
\cite[Theorem 4.1]{draisma-horobet-ottaviani-sturmfels-thomas}. Whenever $X$ is the complexification of a real variety with smooth points, this condition is satisfied. Also note that this condition implies that $T_xX_0\cap (T_xX_0)^{\perp}=\{0\}$ for some $x\in X_0^{\reg}$, so that the ED degree of $X_0$ in $V_0$ is positive as well.
\end{re}

The (proof of the) Main Theorem has the following real analogue. 

\begin{thm}\label{thm_main_real}
Let $V_{\RR}$ be a finite-dimensional real vector space
equipped with a positive definite inner product. Let $K$ be
a Lie group and let $K\to\OO(V_{\RR})$ be an orthogonal
representation. Suppose that $V_{\RR}$ has a linear subspace
$V_{\RR,0}$ such that, for sufficiently general $v_0 \in
V_{\RR,0}$, the space $V_{\RR}$ is the orthogonal direct sum
of $V_{\RR,0}$ and $T_{v_0} Kv_0$. Then every $K$-orbit
intersects $V_{\RR,0}$. Let $X$ be a real $K$-stable closed
subvariety of $V_{\RR}$ and set $X_0:=X \cap V_{\RR,0}$.
Then the number of real critical points of the distance
function to a point is constant along orbits of $K$ and the
set of real critical points on $X$ for a sufficiently
general $v_0\in V_{\RR,0}$ is contained in $X_0$.
\end{thm}

\begin{re}
When we consider an arbitrary $v_0\in V_{\RR,0}$, the space
$$
N_{v_0}=\{v\in V_{\RR}\mid v\perp T_{v_0}Kv_0\}
$$
contains $V_{\RR,0}$, but may be bigger. So while it is still true that the critical points on $X$ for $v_0$ are orthogonal to $T_{v_0}Kv_0$, this does not imply that they lie in $V_{\RR,0}$. However, in this case the stabilizer $K_{v_0}$ acts on $N_{v_0}$ and by \cite[Theorem 2.4]{dadok-kac} this representation again satisfies the conditions of Theorem \ref{thm_main_real} with the subspace $V_{\RR,0}$ of $N_{v_0}$ again playing the same role. In particular, the $K_{v_0}$-orbit of any critical point on $X$ for $v_0$ intersects $V_{\RR,0}$. This allows us to still restrict the search for critical points on $X$ for $v_0$ to $X_0$. Since $K_{v_0}$ preserves the distance to $v_0$, the same is true for closest points on $X$ to $v_0$. 
\end{re}

Apart from proving the Main Theorem, we also classify all orthogonal representations $G\rightarrow\OO(V)$ for which a subspace $V_0$ as in the Main Theorem exists. Theorem~\ref{thm:Polar} below relates this problem, in the case of reductive
$G$, to the classification by Dadok and Kac of so-called {\em polar
representations} \cite{dadok-kac,dadok}.

\begin{de}\label{de:stablepolar}
A complex orthogonal representation $V$ of a reductive algebraic group $G$ is called stable polar when there exists a vector $v\in V$ such that the orbit $Gv$ is closed and maximal-dimensional among all orbits of $G$ and such that the codimension of the subspace $\{x\in V_{\CC}|\g x\subseteq \g v\}$ equals the dimension of $Gv$ where $\g$ is the (complex) Lie algebra of $G$.
\end{de}

\begin{de}
A real orthogonal representation $V_{\RR}$ of a compact Lie group $K$ is called polar when there exists a vector $v\in V_{\RR}$ such that the orbit $Kv$ is maximal-dimensional among all orbits of $K$ and such that $\k u$ is perpendicular to $(\k v)^{\perp}$ for all $u\in (\k v)^{\perp}$ where $\k$ is the (real) Lie algebra of $K$.
\end{de}

\begin{thm}\label{thm:Polar}
Let $V$ be an orthogonal representation of a reductive group $G$. Then the following are equivalent:
\begin{itemize}
\item[(i)] $V$ satisfies the conditions of the Main Theorem; 
\item[(ii)] $V$ is a stable polar representation; and
\item[(iii)] $V$ is the complexification of a polar representation of a maximal compact Lie group $K$ contained in $G$.
\end{itemize} 
\end{thm}

\begin{re}
In (ii), we ask for the representation $V$ to be stable, i.e. for there to exist a $v\in V$ whose orbit is closed and maximal-dimensional among all orbits. This is a notion coming from Geometric Invariant Theory and should not be confused with the notion of a subset $X$ of $V$ being $G$-stable, i.e. having $gX\subseteq X$ for all $g\in G$.

The only places in this paper where the word stable refers to the notion from GIT are in Definition \ref{de:stablepolar} and Theorem \ref{thm:Polar}. 
\end{re}

\begin{re}
Analogously to the equivalence (i)$\Leftrightarrow$(ii) of
Theorem \ref{thm:Polar}, the conditions on $V_{\RR}$ in
Theorem \ref{thm_main_real} are equivalent to $V_{\RR}$
being a polar representation.
\end{re}

In the paper \cite{dadok}, the irreducible real polar representations of compact Lie groups are completely classified, giving us a list of spaces on which our Main Theorem can be applied. We discuss some of these spaces in section \ref{section:examples}. 

\section{Interlude: the closest symmetric matrix with prescribed
eigenvalues} \label{section:symmetric}

Given a symmetric matrix $A \in \RR^{n \times n}$ and given
a sequance of real numbers
$\lambda=(\lambda_1 \leq \lambda_2 \leq \ldots \leq
\lambda_n)$, how does one
find the symmetric matrx $B \in \RR^{n \times n}$ with
spectrum $\lambda$ that minimizes $d_A(B):=\sum_{i,j} (a_{ij}-b_{ij})^2$?

To cast this as an instance of Theorem~\ref{thm_main_real}, take
for $V_{\RR}$ the space of real symmetric matrices acted upon by the group
$K=\OO_n(\RR)$ of orthogonal $n \times n$-matrices via the action $\alpha:
(g,A) \mapsto gAg^T$. The $K$-invariant inner product on
$V_{\RR}$ is
given by
\[ \langle C|D \rangle=\Tr C^T D = \sum_{i,j} c_{ij} d_{ij}.
\]
We claim that the space $V_{\RR,0}$ of {\em diagonal} matrices
has the properties of Theorem~\ref{thm_main_real}.  Indeed, if $D$
is any diagonal matrix with distinct eigenvalues, then differentiating
$\alpha$ and using that the Lie algebra $\k$ of $K$ is the Lie algebra
of skew-symmetric matrices, we find that
\[ T_D K D=\{BD - DB \mid B^T=-B\} \]
is precisely the space of symmetric matrices with zeroes on the diagonal,
i.e., the orthogonal complement of $V_{\RR,0}$.

Let $X$ be the real-algebraic variety in $V_{\RR}$ consisting of matrices
with the prescribed spectrum $\lambda$. Then Theorem~\ref{thm_main_real}
says that, if $A$ lies in $V_0$, so that $A=\diag(\mu_1,\ldots,\mu_n)$,
then the critical points of $d_A$ on $X$ are the same as the critical
points of the restriction of $d_A$ to $X_0:=X \cap V_0$. If the
$\lambda_i$ are distinct, then this intersection consists of $n!$ diagonal
matrices, one for each permutation of the $\lambda_i$. Accordingly,
the ED degree of the complexification of $X$ (the subject of the Main
Theorem) is then $n!$. If the $\lambda_i$ are not distinct but come with
multiplicities $n_1,\ldots,n_k$ adding up to $n$, then the ED degree is
the multinomial coefficient $\frac{n!}{n_1! \cdots n_k!}$. The group
$S_n$ here is the Weyl group from Section~\ref{section:proof}. In
Example~\ref{ex:adjoint} we will see a large class of examples where
the ED degree equals the order of the Weyl group.

Still assuming that $A$ is diagonal, we get a diagonal matrix $B \in
X_0$ closest to $A$ by arranging the $\lambda_i$ in the same order as
the $\mu_i$. To see this, let $\pi\in S_n$ be a permutation. If $\mu_i<\mu_j$ and $\lambda_{\pi(i)}>\lambda_{\pi(j)}$ for some $i,j\in[n]$, then
$$
(\mu_i-\lambda_{\pi(i)})^2+(\mu_j-\lambda_{\pi(j)})^2-(\mu_i-\lambda_{\pi(j)})^2-(\mu_j-\lambda_{\pi(i)})^2=2(\mu_j-\mu_i)(\lambda_{\pi(i)}-\lambda_{\pi(j)})>0
$$
and so $\pi$ cannot minimize $\sum_i (\mu_i-\lambda_{\pi(i)})^2$.

Now when $A$ is not diagonal to begin with, we first compute $g \in
\OO_n(\RR)$ such that $A_0:=g A g^T$ is diagonal, find a diagonal matrix
$B_0$ closest to $A_0$ as above, and then $B:=g^{-1} B_0 g^{-T}$ is a
solution to the original problem. In the same manner, one obtains all
critical points of $d_A$ from those of $d_{A_0}$.

\section{Real compact versus complex reductive}
\label{section:realcomplex}

We will use the correspondence between compact Lie groups and reductive complex linear algebraic groups.

\begin{thm}~
\begin{itemize}
\item[(i)] Any reductive complex algebraic group $G$ contains a maximal compact Lie group. All such subgroups are conjugate and Zariski dense in $G$.
\item[(ii)] Any compact Lie group is maximal in a reductive complex algebraic group, which is unique up to isomorphism.
\end{itemize}
\end{thm}
\begin{proof}
See for example \cite[Subsection 8.7.2]{procesi} and \cite[Section 5.2]{onishchik-vinberg}.
\end{proof}

The following lemma is well known, but included for completeness.

\begin{lm}
The real orthogonal group $\OO_n(\RR)$ is a maximal compact subgroup of the complex orthogonal group $\OO_n$.
\end{lm}
\begin{proof}
Any compact subgroup of $\OO_n$ leaves invariant some Hermitian positive-definite form on $\CC^n$. The only Hermitian positive definite forms that are $\OO_n(\RR)$-invariant are multiples of the standard form. So any compact subgroup of $\OO_n$ containing $\OO_n(\RR)$ is contained in the unitary group $U(n)$. Since $\OO_n(\RR)=\OO_n\cap U(n)$, we see that $\OO_n(\RR)$ is maximal.
\end{proof}

Let $G$ be a reductive linear algebraic group and let $K$ be a maximal compact Lie group contained in $G$. Then the complexification of any real representation of $K$ naturally has the structure of a representation of $G$.

\begin{prop}\label{prop:orthrep=complxiorthrep}
A (complex) representation of $G$ is orthogonal if and only if it is the complexification of a (real) representation of $K$ that is orthogonal with respect to some positive definite inner product.
\end{prop}
\begin{proof}
Let $V$ be an orthogonal real representation of $K$ and let $V_{\CC}$ be its complexification. Extend the inner product $\langle-|-\rangle$ on $V$ to a non-degenerate symmetric bilinear form on $V_{\CC}$. Then $\langle v|w\rangle=\langle gv|gw\rangle$ for all $v,w\in V_{\CC}$ and $g\in K$. So since $K$ is Zariski dense in $G$, we see that $V_{\CC}$ is an orthogonal representation of $G$.

Let $V$ be an orthogonal complex representation of $G$. Then the image of $K$ in $O(V)$ is contained in some maximal compact subgroup $H$ of $O(V)$. Let $W$ be a real subspace of $V$ with $W\otimes\CC=V$ such that the bilinear form on $V$ restricts to a $\RR$-valued positive definite inner product on $W$. Since all maximal compact subgroups of $O(V)$ are conjugate, we see that 
$$
H=gO(W)g^{-1}
$$
for some $g\in O(V)$. Let $V_{\RR}$ be the real vector space $gW$ with inner product $\langle v|w\rangle_{V_{\RR}}=\langle g^{-1}v|g^{-1}w\rangle$ for all $v,w\in V_{\RR}$. Then $V_{\RR}$ is an orthogonal representation of $K$ whose complexification is isomorphic to $V=W\otimes_{\RR}\CC$ by the map $g^{-1}$.
\end{proof}

Let $\g$ be the (complex) Lie algebra of $G$ and let $\k$ be the (real) Lie algebra of $K$. The following theorem is a reformulation of Theorem \ref{thm:Polar}.

\begin{thm}\label{thm:Polar_reform}
Let $V_{\RR}$ be an orthogonal representation of $K$ and let $V_{\CC}$ be its complexification. Then the following are equivalent:
\begin{itemize}
\item[(i)] there exists a (complex) subspace $V_{\CC,0}$ of $V_{\CC}$ such that, for $v_0\in V_{\CC,0}$ sufficientely general, the space $V_{\CC}$ is the orthogonal direct sum of $V_{\CC,0}$ and $\g v_0$;
\item[(ii)] there exists a vector $v\in V_{\CC}$ such that the orbit $Gv$ is closed and maximal-dimensional among all orbits of $G$ and such that the codimension of the subspace $\{x\in V_{\CC}|\g x\subseteq \g v\}$ equals the dimension of $Gv$; and
\item[(iii)] there exists a vector $v\in V_{\RR}$ such that the orbit $Kv$ is maximal-dimensional among all orbits of $K$ and such that $\k u$ is perpendicular to $(\k v)^{\perp}$ for all $u\in (\k v)^{\perp}$.
\end{itemize}
\end{thm}
\begin{proof}~
\begin{itemize}
\item[(ii)$\Rightarrow$(i)] Let $v\in V_{\CC}$ be a vector as in (ii) and take
$$
V_{\CC,0}=\{x\in V_{\CC}|\g x\subseteq \g v\}.
$$
Then for $v_0\in V_{\CC,0}$ sufficiently general, we have
$\g v_0=\g v$. So it suffices to prove that $V_{\CC}$ is the
orthogonal direct sum of $V_{\CC,0}$ and $\g v$. By
\cite[Corollary 2.5]{dadok-kac}, we know that $V_{\CC}$ is the direct sum of $V_{\CC,0}$ (there donoted $c_v$) and $\g v$. We have
$$
\langle V_{\CC,0}|\g v\rangle = -\langle\g V_{\CC,0}|v\rangle=-\langle\g v|v\rangle=\{0\}
$$ 
and therefore the direct sum is orthogonal.
\item[(i)$\Rightarrow$(iii)] Let $V_{\CC,0}$ be a subspace as in (i) and let $U$ be a dense open subset of $V_{\CC,0}$ such that $V_{\CC}$ is the orthogonal direct sum of $V_{\CC,0}$ and $\g w$ for all $w\in U$. Then $GU$ is a dense constructible subset of $V_{\CC}$ and hence contains a dense open subset of $V_{\CC}$. Note that the dimension of the orbit of any element of $GU$ equals the codimension of $V_{\CC,0}$. So since $GU$ is dense in $V_{\CC}$, we see that these orbits must be maximal-dimensional among all orbits of $G$. Since $V_{\RR}$ is dense in $V_{\CC}$, the intersection of $V_{\RR}$ with $GU$ contains a vector $v=gw$ with $g\in G$ and $w\in U$. Since $v\in GU$, we see that 
$$
\dim_{\RR}(Kv)=\dim_{\RR}(\kappa v)=\dim_{\CC}(\g v)=\dim_{\CC}(Gv)
$$
is maximal among the dimensions of all orbits of $K$. The space $V_{\CC}$ is the orthogonal direct sum of $gV_{\CC,0}$ and $\g v$. Therefore we have 
$$
(\k v)^\perp=(\g v)^{\perp}\cap V_{\RR}\subseteq gV_{\CC,0}
$$
and hence for all $u\in (\k v)^{\perp}$, we have 
$$
\langle\k u|(\k v)^{\perp}\rangle\subseteq \langle \g u|gV_{\CC,0}\rangle=\langle g\g g^{-1}u|gV_{\CC,0}\rangle=\langle \g (g^{-1}u)|V_{\CC,0}\rangle=\{0\}.
$$  
\item[(iii)$\Rightarrow$(ii)] Let $v\in V_{\RR}$ be a vector
as in (iii). Since $\langle av|av\rangle=\langle v|v\rangle$
for all $a\in K$, we have $\langle bv|v\rangle+\langle
v|bv\rangle=0$ for all $b\in\k$. So $\langle\k
v|v\rangle=\{0\}$ and $v$ satisfies the condition of 
\cite[Theorem 1.1]{dadok-kac}, because $\langle\g v,v\rangle=\CC\otimes \langle\k v|v\rangle=\{0\}$. Note that the Hermitian form $\langle-,-\rangle$ on $V_{\CC}$ in that theorem is the extension of the inner product on $V_{\RR}$ and that it is not equal to our bilinear form $\langle-|-\rangle$ on $V_{\CC}$. By part (i) of Theorem 1.1, the orbit $Gv$ is closed. Since $K$ is dense in $G$ and since the function $(u\mapsto\dim(Gu))$ is lower semicontinuous, we see that $\dim(Gv)=\dim(Kv)$ is maximal. As stated in the introduction of \cite{dadok-kac}, the dimension of $\{x\in V_{\CC}|\g x\subseteq \g v\}$ is always at most the codimension of a maximal-dimensional orbit of $G$. Since 
$$
\CC\otimes (\k v)^{\perp}\subseteq\CC\otimes \{u\in V_{\RR}|\k u\subseteq\k v\}\subseteq\{x\in V_{\CC}|\g x\subseteq \g v\},
$$
we must have equality.\qedhere
\end{itemize}
\end{proof}

\begin{ex}
Let $G$ be the group $\SL_2(\CC)$ and let $V_{\CC}$ be the irreducible $5$-dimensional representation of $\SL_2(\CC)$. So $V_{\CC}$ is the set of homogeneous polynomials in $x$ and $y$ of degree $4$ and
\begin{eqnarray*}
\sl_2(\CC)&\mapsto&\End_{\CC}(V_{\CC})\\
\begin{pmatrix}a&b\\c&-a\end{pmatrix}&\mapsto&a\left(x\frac{\partial}{\partial x}-y\frac{\partial}{\partial y}\right)+bx\frac{\partial}{\partial y}+cy\frac{\partial}{\partial x}
\end{eqnarray*}
is the corresponding representation of $\sl_2(\CC)$. Let the non-degenerate symmetric bilinear form $\langle-|-\rangle$ on $V_{\CC}$ be given by the Gram matrix
$$
\begin{pmatrix}
&&&&12\\
&&&-3\\
&&2\\
&-3\\
12
\end{pmatrix}
$$
with respect to the basis $x^4,xy^3,x^2y^2,xy^3,y^4$ (obtained by setting $\langle x^4|y^4\rangle=12$ and using $\langle gv|w\rangle=-\langle x|gw\rangle$ for all $v,w\in V_{\CC}$ and $g\in \sl_2(\CC)$). Then $\langle-|-\rangle$ is $\SL_2(\CC)$-invariant. A maximal compact subgroup of $\SL_2(\CC)$ is $K=\SU(2)$. The real subspace
$$
V_{\RR}=\span_{\RR}\left(x^4+y^4, i(x^4-y^4), x^2y^2, xy(x^2-y^2), ixy(x^2+y^2)\right).
$$
of $V_{\CC}$ is $\SU(2)$-stable and has $V_{\CC}$ as its
complexification. See the proofs of 
\cite[Propositions 3 and 5]{itzkowitz-rothman-strassberg} for how $V_{\RR}$ was obtained. We will now check that the three equivalent conditions of the theorem are satisfied.
\begin{itemize}
\item[(i)] Take $V_{\CC,0}=\span_{\CC}(x^4+y^4,x^2+y^2)$. Then $V_{\CC}$ is the orthogonal direct sum of $V_{\CC,0}$ and 
$$
\sl_2(\CC)v_0 =\span_{\CC}\left(x^4-y^4, x^3y, xy^3\right)
$$ 
for all $v_0=a(x^4+y^4)+bx^2y^2$ with $4a^2\neq b^2$. 
\item[(ii)] Take $v=x^4+y^4+x^2y^2$. Then $\dim(\sl_2(\CC) v)=3=\dim(\SL_2(\CC))$. Hence the dimension of $\SL_2(\CC)v$ is maximal. As in the proof of the theorem, we see that the orbit $\SL_2(\CC)v$ is closed and 
$$
\{x\in V_{\CC}|\sl_2(\CC) x\subseteq \sl_2(\CC) v\}=\span_{\CC}(x^4+y^4,x^2+y^2)
$$
has dimension $5-3=2$.
\item[(iii)] Again take $v=x^4+y^4+x^2y^2$. We have
$$
\su(2)=\span_{\RR}\left(\begin{pmatrix}i\\&-i\end{pmatrix},\begin{pmatrix}&-1\\1\end{pmatrix},\begin{pmatrix}&i\\i\end{pmatrix} \right)
$$
and so we see that
$$
\su(2)v=\span_{\RR}\left(i(x^4-y^4),xy(x^2-y^2),ixy(x^2+y^2)\right)
$$
has orthogonal complement
$$
\span_{\RR}\left(x^4+y^4,x^2y^2\right)
$$
and we have $\su(2)u\subseteq\su(2)v$ for all $u$ in this complement.
\end{itemize}
\end{ex}

\section{Proof of the Main Theorem} \label{section:proof}

Let $G \to O(V)$ be an orthogonal representation as in Section \ref{section:mainresults}. Let $X$ be a $G$-stable closed subvariety of $V$. We assume the conditions of the Main Theorem. Note that if we replace $G$ by its unique irreducible component $G^{\circ}$ that contains the identity element, the conditions of the Main Theorem are still satisfied, because $G^{\circ}$ has finite index in $G$. So we may assume that $G$ is irreducible. This implies that all irreducible components of $X$ are also $G$-stable.

\begin{lm} \label{lm:GV0}
The set $GV_0$ is dense in $V$.
\end{lm}
\begin{proof}
The derivative of the multiplication map $G\times V_0\to V$ at a (smooth) point $(1,v_0)$ equals the map 
\begin{eqnarray*}
\g\oplus V_0 &\to& V\\
(A,u_0) &\mapsto& Av_0+u_0
\end{eqnarray*}
and has image $\g v_0 + V_0$, which by assumption equals $V$ for sufficienly general $v_0\in V_0$. Hence the derivative is surjective at $(1,v_0)$ for some $v_0\in V_0$. Therefore the multiplication map is dominant and its image $GV_0$ is dense in $V$.
\end{proof}

\begin{lm} \label{lm:Move}
For elements $g \in G$ and $u \in V$, the ED critical points for $gu$ are obtained from those of $u$ by applying $g$.
\end{lm}
\begin{proof}
Let $x$ be a point on $X$. The element $g\in G$ acts linearly and preserves $X$ and $X^{\reg}$. The derivative of the isomorphism $X\rightarrow X, y\mapsto gy$ at $x$ is the isomorphism $T_xX\rightarrow T_{gx}X, w\mapsto gw$. So since $g$ also preserves the billinear form, we have $u-x \perp T_x X$ if and only if $gu-gx\perp T_{gx}X$.
\end{proof}

\begin{lm} \label{lm:genX0isXreg}
A sufficiently general $x_0\in X_0$ lies both in $X_0^{\reg}$ and in $X^{\reg}$.
\end{lm}
\begin{proof}
A sufficiently general point on $X_0$ lies in $X_0^{\reg}$. Since $GX_0$ is constructible and dense in $X$, it contains a $G$-stable dense open subset $U$ of $X^{\reg}$. The intersection of $U$ with $X_0$ is dense in $X_0$. Hence a sufficiently general point on $X_0$ lies in $X^{\reg}$. 
\end{proof}

Define the Weyl group $W$ by
$$
W=N_G(V_0)/Z_g(V_0)=\{g\in G|gV_0=V_0\}/\{g\in G|gw=w\forall w\in V_0\}.
$$
Then the finite group $W$ acts naturally on $V_0$. Consider the set $S$ of $G$-stable closed subvarieties $Y$ of $V$ such that $G(Y\cap V_0)$ is dense in $Y$ and the set $R$ of $W$-stable closed subvarieties of $V_0$. Consider the maps
$$
\begin{array}{ccccccccc}
\varphi\colon &S&\rightarrow&R&\mbox{and}&\psi\colon&R&\rightarrow&S\\
&Y&\mapsto&Y\cap V_0&&&Z&\mapsto&\overline{GZ}
\end{array}
$$
between these two sets.

\begin{lm}\label{lm:BijCor}
The bijective maps $\varphi$ and $\psi$ are mutual inverses.
\end{lm}
\begin{proof}
Since $S$ consists of the $G$-stable closed subvarieties $Y$
of $V$ such that $Y$ equals the closure of $G(Y\cap V_0)$ in
$V$, we see that $\psi\circ\varphi$ is the identity map on
$S$. Let $Z$ be a $W$-stable closed subvariety of $V_0$. It
is clear that $Z\subseteq\varphi(\psi(Z))$ and we will show
that in fact $\varphi(\psi(Z))=Z$ holds. Since $Z$ is
$W$-stable and $W$ is finite, the variety $Z$ is defined by
$W$-invariant polynomials $f_1,\dots,f_n\in\CC[V_0]^W$. By
\cite[Theorem 2.9]{dadok-kac}, there exists $G$-invariant polynomials $g_1,\dots,g_n\in\CC[V]^G$ such that $f_i$ is the restriction of $g_i$ to $V_0$ for all $i\in\{1,\dots,n\}$. Since $g_1,\dots,g_n$ are $G$-invariant and $g_1(z)=\dots=g_n(z)=0$ for all $z\in Z$, we see that (the closure of) $GZ$ is contained in the zero set of the ideal generated by $g_1,\dots,g_n$. Hence
$$
\varphi(\psi(Z))=\overline{GZ}\cap V_0
$$
is contained in the zero set of the ideal generated by the restrictions of $g_1,\dots,g_n$ to $V_0$. This zero set is $Z$ and hence $\varphi(\psi(Z))\subseteq Z$. So we see that $\varphi\circ\psi$ is the identity map on $R$.
\end{proof}

\begin{lm} \label{lm:genX0hasNiceTangent}
A sufficiently general $x_0\in X_0$ satisfies $T_{x_0}X=T_{x_0}X_0 + T_{x_0}Gx_0$.
\end{lm}
\begin{proof}
By Lemma \ref{lm:genX0isXreg}, we see that sufficiently
general points of $X_0$ are contained in at most one
irreducible component of $X$. Therefore each irreducible
component of $X_0$ is contained in precisely one irreducible
component of $X$. Let $Y$ be an irreducible component of $X$
and let $Z_1,\dots,Z_k$ be the irreducible components of
$X_0$ contained in $Y$. Then the Weyl group $W$ acts on the
set $\{Z_1,\dots,Z_k\}$. Since $GX_0$ is dense in $X$, we
see that $G(Z_1\cup\dots\cup Z_k)$ must be dense in $Y$. So
$GZ_i$ must be dense in $Y$ for some $i\in\{1,\dots,k\}$. By
the previous lemma, for this $i$ we have 
$$
Z_1\cup\dots\cup Z_k=Y\cap V_0=\bigcup_{g\in W}gZ_i
$$
and hence $W$ must act transitively on $\{Z_1,\dots,Z_k\}$.
In particular, we see that $GZ_j$ is in fact dense in $Y$ for all
$j\in\{1,\dots,k\}$.

Take $Z=Z_j$ for any $j\in\{1,\dots,k\}$. Then the multiplication map $G\times Z\rightarrow Y$ is dominant and $G$-equivariant when we let $G$ act on itself by left multiplication. Therefore its derivative at $(1,z)$ is surjective for $z\in Z$ sufficiently general. This means that $T_{z}Y=T_{z}Z + T_{z}Gz$ for $z\in Z$ suffciently general. Since this holds for all components $Z$ of $X_0$, we see that $T_{x_0}X=T_{x_0}X_0 + T_{x_0}Gx_0$ for $x_0\in X_0$ suffciently general.
\end{proof}

\begin{lm}\label{lm:CritImpliesGeneral}
Let $Y$ be a closed subvariety in a complex affine space $V$. Let $U$ be a dense open subset of $Y$ and let $Z$ be its complement in $Y$. Then for $v\in V$ sufficiently general, all ED critical points $y\in Y$ for $v$ lie in $U$.
\end{lm}
\begin{proof}
See the proof of \cite[Lemma 4.2]{drusvyatskiy-lee-ottaviani-thomas}.
\end{proof}

\begin{lm}\label{lm:X0critIsXcrit}
Let $v_0\in V_0$ be sufficiently general. Then any ED critical point on $X_0$ for $v_0$ is an ED critical point on $X$ for $v_0$.
\end{lm}
\begin{proof}
By combining the previous three lemmas, we may assume that all ED critical points $x_0\in X_0$ for $v_0$ are not only elements of $X_0^{\reg}$ but also of $X^{\reg}$ and that they satisfy $T_{x_0}X=T_{x_0} X_0+T_{x_0} Gx_0$. Let $x_0$ be an ED critical points of $v_0$. Then $v_0-x_0 \perp T_{x_0} X_0$ by criticality and $v_0-x_0 \in V_0 \perp T_{x_0} G x_0$ by the conditions of the Main Theorem (here we do not need that $T_{x_0} G x_0$ is the orthogonal complement of $V_0$---this may not be true---but only that it is contained in that complement). We see that 
$$
v_0-x_0 \perp T_{x_0} X_0+T_{x_0} G x_0=T_{x_0}X
$$
and hence $x_0$ is an ED critical point on $X$ for $v_0$.
\end{proof}

\begin{lm}\label{lm:XcritIsX0crit}
Let $v_0\in V_0$ be sufficiently general. Then any ED critical point on $X$ for $v_0$ is an ED critical point on $X_0$ for $v_0$.
\end{lm}
\begin{proof}
Let $x \in X$ be an ED critical point for $v_0$. Then in particular $v_0-x\perp T_x Gx=\g x$. Together with $x\perp\g x$, which holds by orthogonality of the representation, this implies that $v_0 \perp \g x$. Using once more the orthogonality of the representation, we see that $\langle x|\g v_0\rangle=-\langle\g x|v_0\rangle=\{0\}$. So $x \perp T_{v_0} G v_0$. Since $v_0$ is sufficiently general in $V_0$, the vector space $V$ is the orthogonal direct sum of $V_0$ and $T_{v_0}Gv_0$ and therefore $x$ is an element of $V_0$. So since also $x \in X$, we have $x \in X_0$. Since $v_0-x \perp T_x X \supseteq T_x X_0$, we find that $x\in X_0$ is an ED critical point for $v_0$. 
\end{proof}

\begin{proof}[Proof of the Main Theorem]
By Lemmas~\ref{lm:GV0} and~\ref{lm:Move} we may assume that the sufficiently general point on $V$ is in fact a sufficiently general point $v_0$ on $V_0$. The previous two lemmas now tell us that the ED critical points for $v_0$ on $X$ and on $X_0$ are the same. Hence the ED degrees of $X$ in $V$ and $X_0$ in $V_0$ are equal. 
\end{proof}

\begin{ex} \label{ex:adjoint}
Let $G$ be a complex semisimple algebraic group acting on
its Lie algebra $V=\g$ by conjugation, let $V_0$ be a Cartan
subalgebra of $\g$ and let $W$ be the Weyl group. In Section
\ref{section:examples}, we show that $V$ satisfies the
conditions of the Main Theorem. Suppose $X$ is the closed
$G$-orbit of a sufficiently general point $v\in V_0$. Then
the intersection $X_0=X\cap V_0$ is a single $W$-orbit by
\cite[Theorem 2.8]{dadok-kac}. Since $X_0$ is the $W$-orbit of a sufficiently general point of $V_0$, it is a set of size $\# W$. So the ED degree of $X$ equals $\#W$.

Since $v$ is sufficienlty general, the codimension of $X$ in
$V$ equals the dimension of $V_0$. So the degree of the
variety $X$, i.e. the cardinality of $X\cap V'$ for a
sufficiently general subspace $V'$ of $V$ with
$\dim(V')=\dim(V_0)$, is at least the cardinality of $X\cap
V_0$, which is the ED degree of $X$. Let $f_1,\dots,f_n$ be
a set of invariant polynomial generating the algebra
$\CC[V]^G$. Then, since $X$ is a closed $G$-orbit, we see
that $X$ is defined by the equations $f_1=f_1(v), \dots,
f_n=f_n(v)$. Therefore the degree of $X$ is at most the
product of the degrees of $f_1$, \dots, $f_n$. By Theorem
\cite[Theorem 3.19]{humphreys}, this product equals the size of the Weyl group $W$. So the degree and ED degree of $X$ are equal.
\end{ex}

\section{Testing for the conditions of the Main Theorem} \label{section:testing}

In the paper \cite{dadok}, the irreducible polar
representations of compact Lie groups are completely
classified. Now suppose we have a not-necessarily
irreducible orthogonal representation $V$ of a reductive
algebraic group $G$. We would like to be able to test
whether $V$ satisfies the conditions of the Main Theorem. In
\cite[Section 2]{dadok-kac}, some methods are given. In this section, we describe one more such method.

\begin{lm}
For sufficiently general $v \in V$, the tangent space $\g v$ of $v$ to its orbit is maximal-dimensional and non-degenerate with respect to the bilinear form.
\end{lm}
\begin{proof}
By Proposition \ref{prop:orthrep=complxiorthrep}, we know that $V$ is the complexification of a real subspace $V_{\RR}$ and that the $G$-invariant bilinear form $\langle-|-\rangle$ is the extension of a positive definite inner product on $V_{\RR}$. Since $V_{\RR}$ is dense in $V$ and the set of $v\in V$ such that the dimension of $\g v$ is maximal is open and dense, there is an element $w$ in the intersection. Note that $\g w=\k w\otimes_{\RR}\CC$. Since $\langle-|-\rangle$ is an inner product on $V_{\RR}$, its restriction to $\k w$ is non-degenerate. Therefore the restriction of $\langle-|-\rangle$ to $\g w$ is non-degenerate as well.

Pick $\varphi_1,\dots,\varphi_n\in\g$ such that $\varphi_1w,\dots,\varphi_nw$ form a basis of $\g w$. Then the set of $v\in V$ such that $\varphi_1v,\dots,\varphi_nv$ are linearly independent and the restriction of $\langle-|-\rangle$ to their span is non-degenerate is an non-empty open subset of $V$. Since the dimension of $\g w$ is maximal, we see that for every element $v$ in this set, the tangent space $\g v$ of $v$ to its orbit is spanned by $\varphi_1v,\dots,\varphi_nv$. So for sufficiently general $v\in V$, the vector space $\g v$ is non-degenerate with respect to the bilinear form.
\end{proof}

\begin{lm}\label{lm:Transitive}
A subspace $V_0$ as in the Main Theorem exists if and only if for sufficiently general $v \in V$ and for all $u_1,u_2 \perp \g v$ we have $u_1 \perp\g u_2$.
\end{lm}
\begin{proof}
Suppose such a subspace $V_0$ exists. Let $v \in V$ be sufficiently general and let $u_1,u_2 \perp \g v$. For $v_0\in V_0$ sufficiently general and for all $g\in G$, the vector space $V$ is the orthogonal direct sum of $gV_0$ and $g\g v_0=\g(gv_0)$. Since $GV_0$ contains an open dense subset of $V$, we may assume that $v=gv_0$ for such $v_0$ and $g$. So we see that $u_1,u_2\in gV_0$. We have $V_0\perp \g u$ for all $u\in V_0$. Therefore we have $gV_0\perp\g u$ for all $u\in gV_0$ and hence $u_1\perp\g u_2$. 

Let $v\in V$ be such that the $\g v$ is maximal-dimensional, the restriction of the $G$-invariant bilinear form $\langle-|-\rangle$ to $\g v$ is non-degenerate and $u_1 \perp\g u_2$ for all $u_1,u_2 \perp \g v$. Let $V_0$ be the orthogonal complement of $\g v$. Then we see that $V_0$ is perpendicular to $\g v_0$ for all $v_0\in V_0$. Let $v_0\in V_0$ be sufficiently general. Then we have $\g v_0=\g v$ and hence $V$ is the orthogonal direct sum of $V_0$ and $\g v_0$. 
\end{proof}

\begin{lm}\label{lm:RationalCoeff}
Let $W$ be a finite-dimensional complex vector space, let 
$$
f_1,\dots,f_m\colon V\rightarrow W
$$
be linear maps and let $w\in W$ be an element. Then the following are equivalent:
\begin{itemize}
\item[(i)] For $v\in V$ sufficiently general, we have $w\in\span_{\CC}(f_1(v),\dots,f_m(v))$.
\item[(ii)] We have $1\otimes w\in\span_{\CC(V^*)}(f_1,\dots,f_m)\subset \CC(V^*)\otimes_{\CC}W$.
\end{itemize}
\end{lm}
\begin{proof}
Suppose that 
$$
1\otimes w = c_1f_1+\dots+c_mf_m
$$
for some $c_1,\dots,c_m\in \CC(V^*)$. Then 
$$
w=c_1(v)f_1(v)+\dots+c_m(v)f_m(v)
$$
for the dense open subset of $V$ consisting of all $v$ where $c_1,\dots,c_m$ can be evaluated.

For the converse, suppose that for $v\in V$ sufficiently general we know that $w$ is contained in the span of $f_1(v)$,\dots, $f_m(v)$. We may assume that $f_1(v),\dots,f_m(v)$ are linearly independent for $v\in V$ sufficiently general. Let $v\in V$ be such that $f_1(v),\dots,f_m(v)$ are linearly independent and $w$ is contained in their span. Choose $w_1,\dots,w_k\in W$ such that $f_1(v),\dots,f_m(v),w_1,\dots,w_k$ form a basis of $W$. Now note that $f_1(v),\dots,f_m(v),w_1,\dots,w_k$ form a basis of $W$ for $v\in V$ sufficiently general. By choosing a basis, we may assume that $W=\CC^{n+k}$. This gives us a morphism 
\begin{eqnarray*}
\varphi\colon V&\rightarrow&\CC^{(m+k)\times(m+k)}\\
v&\mapsto&\begin{pmatrix}f_1(v)&\dots&f_m(v)&w_1&\dots&w_k\end{pmatrix}
\end{eqnarray*}
such that $\varphi(v)$ is invertible for $v\in V$ sufficiently general. Consider the coefficients of $\varphi(v)$ as elements of the field $\CC(V^*)$ of rational functions on $V$. Then the matrix $\varphi(v)$ is invertible and $c(v)=\varphi(v)^{-1}w$ is a vector with coefficients in $\CC(V^*)$. We have
$$
w=\varphi(v)c(v)= c_1(v)f_1(v)+\dots+c_m(v)f_m(v)+c_{n+1}(v)w_1+\dots+c_{n+k}(v)w_k
$$
for $v\in V$ sufficiently. Since we also know that $f_1(v),\dots,f_m(v),w_1,\dots,w_k$ form a basis and that $w$ is contained in the span of $f_1(v),\dots,f_m(v)$ for $v\in V$ sufficiently general, we see that $c_{n+1},\dots,c_{n+k}$ all must be equal to the zero function. Hence
$$
1\otimes w=c_1f_1+\dots+c_mf_m
$$
is contained in the span of $f_1,\dots,f_m$ inside $\CC(V^*)\otimes W$.
\end{proof}

Now we combine the previous two lemmas to reduce checking the existence of $V_0$ to a linear algebra problem over $\CC(V^*)$. Take $U=W=V$ and consider $U$ and $W$ as affine spaces. Let $\varphi_1,\dots,\varphi_n$ form a basis of $\g$. By Lemma \ref{lm:Transitive}, we know that the representation $V$ satisfies the conditions of the Main Theorem if and only if, for $v\in V$ sufficiently general, the variety in $U\times W$ given by the linear equations 
$$
\langle u|\varphi_iv\rangle=\langle w|\varphi_iv\rangle=0,\quad i=1,\dots,n
$$
is contained in the variety given by the equations $\langle u|\varphi_j w\rangle=0$ for $j=1,\dots,n$. The latter holds if and only if the polynomials $\langle u|\varphi_j w\rangle$ are contained in the ideal $I$ of the coordinate ring $\CC[U\times W]$ generated by $\langle u|\varphi_iv\rangle$ and $\langle w|\varphi_iv\rangle$ for $i=1,\dots,n$. The polynomial $\langle u|\varphi_j w\rangle$ is homogeneous of degree $2$. So for a fixed $v\in V$, it is contained in $I$ if and only if
$$
\langle u|\varphi_j w\rangle\in \span_{\CC}\left(\CC[U\times W]_{(1)}\cdot\left\{\left.\langle u|\varphi_iv\rangle,\langle w|\varphi_iv\rangle\right|i=1,\dots,n\right\}\right).
$$
So by Lemma \ref{lm:RationalCoeff}, we see that $V$ satisfies the conditions of the Main Theorem if and only if
$$
\langle u|\varphi_j w\rangle\in \span_{\CC(V^*)}\left(\CC[U\times W]_{(1)}\cdot\left\{\left.\langle u|\varphi_iv\rangle,\langle w|\varphi_iv\rangle\right|i=1,\dots,n\right\}\right)
$$
for all $j\in\{1,\dots,n\}$. The latter condition can be checked efficiently on a computer, requiring as input the bilinear form $\langle-|-\rangle$ and the images in $\End(V)$ of a basis of $\g$.

\section{examples} \label{section:examples}

In this section we highlight some of the families of polar representations coming from \cite{dadok}. We also point out how some of these families are related by means of slice representations as defined in \cite{dadok-kac}. Our Main Theorem can be applied to each of these families, thus generalizing 
\cite[Theorem 4.11]{drusvyatskiy-lee-ottaviani-thomas}.

\begin{re}
A representation $V$ of a group $G$ satisfies the conditions of the Main Theorem if and only if the direct sum of $V$ with the trivial representation does.
\end{re}

\begin{re}
Let $V$ be the orthogonal direct sum of two representations $V_1$ and $V_2$ of $G$. Then if $V$ satisfies the conditions of the Main Theorem, so do $V_1$ and $V_2$.
\end{re}

\subsection{Adjoint representations} 

Let $G$ be a complex semisimple algebraic group acting on its Lie algebra $\g$ by conjugation. This representation is orthogonal with respect to the Killing form $B$ on $\g$ defined by
$$
B(v,w)=\Tr(\ad v\ad w)
$$
for $v,w\in\g$. Since $G$ is semisimple, we know that $B$ is non-degenerate. Since $G$ acts by conjugation, the tangent space of a point $v\in\g$ to its orbit equals $[\g,v]$. We have
$$
B(w,[\g,v])=-B([w,v],\g)
$$
for all $w,v\in\g$. So $w\perp\g v$ if and only if $[v,w]=0$. Let $h\in\g$ and suppose $[h,v]\perp\g v$. Then $h\in \ker(\ad v)^2=\ker(\ad v)$ and hence $[h,v]=0$. Hence $\g v$ is non-degenerate. Let $V_0$ be a Cartan subalgebra of $\g$. Let $v\in V_0$ be sufficiently general and let $u_1,u_2\perp\g v$. Then $V_0=C_{\g}(V_0)=C_{\g}(v)$. So we have $u_1,u_2\in V_0$ and hence $u_1\perp \g u_2$.

\subsection{Standard representations of groups of type $B$ and $D$}

Let $n$ be a positive integer and let $G$ be the orthogonal group $\OO(n)$ acting on $\CC^n$ with the standard form. Let $V_0$ be the subspace of $V$ spanned by the first basis vector $e_1$. For all $v\in V_0$ non-zero, we have $\g v=\{Ae_1|A\in\gl_n\mbox{ skew-symmetric}\}=\span(e_2,\dots,e_n)=V_0^{\perp}$.

\subsection{Representations of groups of type $B$ and $D$ of highest weight $2\lambda_1$} 

Let $n$ be a positive integer and let $G$ be the orthogonal group $\SO(n)$ acting on the vector space $V$ of symmetric $n\times n$ matrices with trace zero by conjugation. The bilinear form given by
$$
\langle A|B\rangle=\Tr(A^TB)
$$
for $A,B\in V$ is non-degenerate and $\SO(n)$-invariant. Let $V_0$ be the subspace of $V$ consisting of all diagonal matrices. For all $D\in V_0$ with pairwise distinct entries on the diagonal, we have $\g\cdot D=\{AD-DA|A\in\gl_n\mbox{ skew-symmetric}\}=V_0^{\perp}$.

\subsection{Tensor products of two standard representations of groups of type $B$ and $D$} 

Let $n\leq m$ be positive integers and let $G$ be the group $\OO(n)\times\OO(m)$ acting on $n\times m$ matrices by left and right multiplication. The bilinear form given by
$$
\langle A|B\rangle=\Tr(A^TB)
$$
for $A,B\in \CC^{n\times m}$ is non-degenerate and $G$-invariant. The subspace $V_0$ of $\CC^{n\times m}$ consisting of diagonal matrices satisfies the conditions of the Main Theorem.

\begin{re}
Consider the matrix $v=(I_n~0)\in\CC^{n\times m}$. The stabilizer of $v$ equals
$$
G_v=\left\{\left.\left(g,\begin{pmatrix}g&0\\0&h\end{pmatrix} \right)\right|g\in\OO(n),h\in\OO(m-n)\right\}
$$
and the orthogonal complement of $\g v$ equals the set of matrices of the form $(A~0)$ where $A$ is a symmetric $n\times n$ matrix. Ignoring the trivial action from $\OO(m-n)$, we see that the slice representation of the element $v$ is the direct sum of the representation from the previous subsection and the trivial representation.
\end{re}

\subsection{Second alternating powers of standard representations of groups of type $C$}

Let $n$ be a positive integer and let $G$ be the symplectic group 
$$
\Sp(n)=\left\{A\in\GL_{2n}\left|A\begin{pmatrix}0&I_n\\-I_n&0\end{pmatrix}A^T=\begin{pmatrix}0&I_n\\-I_n&0\end{pmatrix}\right.\right\}
$$
acting on the second alternating power $\Lambda^2 \CC^{2n}$ of the standard representation. The Lie algebra of $\Sp(n)$ equals
\begin{eqnarray*}
\sp(n)&=&\left\{A\in\gl_{2n}\left|A\begin{pmatrix}0&I_n\\-I_n&0\end{pmatrix}+\begin{pmatrix}0&I_n\\-I_n&0\end{pmatrix}A^T\right.\right\}\\
&=&\left\{\left.\begin{pmatrix}X&Y\\Z&\Theta\end{pmatrix}\in\gl_{2n}\right|\begin{array}{c}Y=Y^T,~Z=Z^T\\X+\Theta^T=0\end{array} \right\}.
\end{eqnarray*}
The $\Sp(n)$-invariant skew-symmetric form on $\CC^{2n}$ induces the bilinear form on $\Lambda^2\CC^{2n}$ given by
$$
\langle v\land w|x\land y\rangle = v^T\begin{pmatrix}0&I_n\\-I_n&0\end{pmatrix}x\cdot w^T\begin{pmatrix}0&I_n\\-I_n&0\end{pmatrix}y-v^T\begin{pmatrix}0&I_n\\-I_n&0\end{pmatrix}y\cdot w^T\begin{pmatrix}0&I_n\\-I_n&0\end{pmatrix}x
$$
for $v,w,x,y\in\CC^{2n}$. This form is symmetric, non-degenerate and $\Sp(2n)$-invariant. Let $V_0$ be the subspace of $\Lambda^2\CC^{2n}$ spanned by $e_i\land e_{n+i}$ for $i=1,\dots,n$. Then for any linear combination $v$ of $e_1\land e_{n+1},\dots,e_n\land e_{2n}$ with only non-zero coefficients, the vector space $\Lambda^2\CC^{2n}$ is the orthogonal direct sum of $V_0$ and $\g v$.

\begin{re}
The paper \cite{dadok} tells us that $\Lambda^2\CC^{2n}$ is isomorphic to $\gl_{2n}/\sp(n)$ acted on by $\Sp(n)$ by conjugation. In this case, the subspace $V_0$ of consists of matrices of the form
$$
\begin{pmatrix}D&0\\0&D\end{pmatrix}
$$
with $D\in\gl_n$ diagonal.
\end{re}

\subsection{Tensor products of two standard representations of groups of type $C$} 

Let $n\leq m$ be positive integers and let $G$ be the group $\Sp(n)\times\Sp(m)$ acting on $2n\times 2m$ matrices by left and right multiplication. The bilinear form is defined by
$$
\langle A|B\rangle=\Tr\left(\begin{pmatrix}0&I_n\\-I_n&0\end{pmatrix}A\begin{pmatrix}0&I_m\\-I_m&0\end{pmatrix}B^T\right)
$$
for all $A,B\in\CC^{2n\times 2m}$. This form is symmetric, non-degenerate and $G$-invariant. Let $V_0$ be the subspace of $V$ consisting of matrices of the form
$$
\begin{pmatrix}D&0&0&0\\0&0&D&0\end{pmatrix}
$$
where $D$ is a diagonal $n\times n$ matrix. Then for every invertible diagonal $n\times n$ matrix $D$ whose squares of diagonal entries are pairwise distinct, the vector space $\CC^{2n\times 2m}$ is the orthogonal direct sum of $V_0$ and
$$
\g\cdot \begin{pmatrix}D&0&0&0\\0&0&D&0\end{pmatrix}=\sp(n) \begin{pmatrix}D&0&0&0\\0&0&D&0\end{pmatrix}+ \begin{pmatrix}D&0&0&0\\0&0&D&0\end{pmatrix}\sp(m).
$$

\begin{re}
The slice representation of
$$
\begin{pmatrix}I_n&0&0&0\\0&0&I_n&0\end{pmatrix}
$$
is a representation of $\Sp(n)\times\Sp(m-n)$ where the second factor acts trivially. Ignoring this factor, the slice representation is isomorphic to the representation $\gl_{2n}/\sp(n)$ from the previous remark.
\end{re}

\subsection{Direct sums of standard representations of groups of type $A$ and their duals} 

Let $n$ be a positive integer and let $G$ be the group $\SL_n$ acting on $\CC^n\oplus\CC^n$ by 
$$
g\cdot(v,w)=(gv,g^{-T}w)
$$
for all $g\in\SL_n$ and $v,w\in\CC^n$. Let the bilinear form be given by
$$
\langle (v,w)|(x,y)\rangle=v^Ty+x^Tw.
$$
for all $v,w,x,y\in\CC^n$. This form is symmetric, non-degenerate and $\SL_n$-invariant. Let $V_0$ be the subspace of $\CC^n\oplus\CC^n$ spanned by $(e_1,e_1)$. Then $\CC^n\oplus\CC^n$ is the orthogonal direct sum of $V_0$ and $\g\cdot v$ for all non-zero $v\in V_0$.

\subsection{Direct sums of representations of groups of type $A$ of highest weight $2\lambda_1$ and their duals} 

Let $n$ be a positive integer and let $G$ be the group $\GL_n$ acting on the vector space $V$ of pairs of symmetric $n\times n$ matrices by $g\cdot (A,B)=(gAg^T,g^{-T}Bg^{-1})$ for all $g\in\GL_n$ and $(A,B)\in V$. Let the bilinear form on $V$ be given by
$$
\langle(A,B)|(C,D)\rangle = \Tr(AD+BC)
$$
for all symmetric matrices $A,B,C,D\in\gl_n$. Let $V_0$ be the subspace
$$
\{(D,D)|D\in\gl_n\mbox{ diagonal}\}
$$
of $V$. Then for every invertible diagonal $n\times n$ matrix $D$ whose squares of diagonal entries are pairwise distinct, the vector space $V$ is the orthogonal direct sum of $V_0$ and
$$
\g\cdot (D,D)=\{(AD+DA^T,-A^TD-DA)|A\in\gl_n\}.
$$

\begin{re}
The slice representation of $(I_n,I_n)$ is isomorphic to the set of symmetric $n\times n$ matrices acted on by $\OO_n$ with conjugation.
\end{re}

\subsection{Direct sums of representations of groups of type $A$ of highest weight $\lambda_2$ and their duals} 

Let $n$ be a positive integer and let $G$ be the group $\GL_n$ acting on the vector space $V$ of pairs of skew-symmetric $n\times n$ matrices by $g\cdot (A,B)=(gAg^T,g^{-T}Bg^{-1})$ for all $g\in\GL_n$ and $(A,B)\in V$. Let the bilinear form on $V$ be given by
$$
\langle(A,B)|(C,D)\rangle = \Tr(AD+BC)
$$
for all skew-symmetric matrices $A,B,C,D\in\gl_n$. Let $V_0$ be the subspace
$$
\left\{\left.\left(\begin{pmatrix}0&E\\-E&0\end{pmatrix},\begin{pmatrix}0&E\\-E&0\end{pmatrix}\right)\right|E\in\gl_{n/2}\mbox{ diagonal}\right\}
$$
of $V$ if $n$ is even and the subspace
$$
\left\{\left.\left(\begin{pmatrix}0&0&E\\0&0&0\\-E&0&0\end{pmatrix},\begin{pmatrix}0&0&E\\0&0&0\\-E&0&0\end{pmatrix}\right)\right|E\in\gl_{(n-1)/2}\mbox{ diagonal}\right\}
$$
of $V$ if $n$ is odd. Then for every invertible diagonal $\lfloor n/2\rfloor \times \lfloor n/2\rfloor$ matrix $E$ whose squares of diagonal entries are pairwise distinct, the vector space $V$ is the orthogonal direct sum of $V_0$ and the tangent space at the corresponding element of $V_0$ to its orbit.

\begin{re}
Suppose $n$ is even. Then the slice representation of
$$
\begin{pmatrix}0&I_{n/2}\\-I_{n/2}&0\end{pmatrix}
$$
is isomorphic to the representation $\Lambda^2\CC^n$ of $\Sp(n/2)$.
\end{re}

\subsection{Direct sums of standard representations of groups of type $C$ and their duals}

Let $n$ be a positive integer and let $G$ be the group $\Sp(n)$ acting on the vectorspace $\CC^{2n}\oplus\CC^{2n}$ with the form given by
$$
\langle(v,w)|(x,y)\rangle = v^T\begin{pmatrix}o&I_n\\-I_n&0\end{pmatrix}y+x^T\begin{pmatrix}o&I_n\\-I_n&0\end{pmatrix}w
$$
for all $v,w,x,y\in\CC^{2n}$. Let $V_0$ be the subspace of $\CC^{2n}\oplus\CC^{2n}$ spanned by some $(v,w)$ with $v_i,w_i\neq0$ for all $i$ and $v_iw_j-v_jw_i\neq0$ for all $i\neq j$. Then it follows from the following lemma that $\CC^{2n}\oplus\CC^{2n}$ is the orthogonal direct sum of $V_0$ and the tangent space at any non-zero element $V_0$ of its orbit.

\begin{lm}
Let $v,w,x,y\in\CC^m$ be such that $v_i,w_i\neq0$ for all $i$ and $v_iw_j\neq v_jw_i$ for all $i\neq j$. Then $v^TSy=x^TSw$ for all symmetric $m\times m$ matrices $S$ if and only if $(x,y)=\lambda(v,w)$ for some $\lambda\in\CC$.
\end{lm}

\subsection{Tensor products of two direct sums of standard representations of groups of type $A$ and their duals}

Let $n\leq m$ be positive integers and let $\GL_n\times\GL_m$ act on $\CC^{n\times m}\oplus\CC^{n\times m}$ by
$$
(g,h)(A,B)=\left(gAh^T,g^{-T}Bh^{-1}\right)
$$
for all $g\in\GL_n$, $h\in\GL_m$ and $A,B\in\CC^{n\times m}$. Let the bilinear form be given by
$$
\langle (A,B)|(C,D)\rangle = \Tr\left(A^TD+C^TB\right)
$$
for all $A,B,C,D\in\CC^{n\times m}$. Let $V_0$ be the subspace 
$$
\left\{\left((D~0),(D~0)\right)\left|D\in\gl_n\mbox{diagonal}\right.\right\}
$$
of $V$. Then for all invertible diagonal $n\times n$ matrices $D$ whose squares of diagonal entries are pairwise distinct, the vector space $\CC^{n\times m}\oplus\CC^{n\times m}$ is the orthogonal direct sum of $V_0$ and
$$
\g \left((D~0),(D~0)\right)=\left\{\left.\left((AD~0)+(D~0)B^T,(-A^TD~0)-(D~0)B\right)\right|A\in\gl_n,B\in\gl_m\right\}.
$$

\begin{re}
The slice representation of the pair $((I_n~0),(I_n~0))$ is a representation of $\GL_n\times\GL_m$ where the second factor acts trivially. Ignoring this factor, we get the adjoint representation of $\GL_n$. 
\end{re}

\end{document}